\documentclass[10pt]{amsart}

\usepackage[initials]{amsrefs}
\usepackage{amssymb,amsmath,amsfonts,latexsym,amsthm,paralist,graphicx}

\def\N{\mathbb N}
\def\C{\mathbb C}
\def\R{\mathbb R}

\def\Q{\mathbb Q}

\def\ve{\varepsilon}

\newtheorem{theorem}{Theorem}[section]

\newtheorem{proposition}[theorem]{Proposition}
\newtheorem{lemma}[theorem]{Lemma}

\theoremstyle{definition}

\newtheorem{remark}[theorem]{Remark}

\newtheorem{definition}[theorem]{Definition}

\begin{document}

\title[\tiny Large algebras of singular functions vanishing on prescribed sets]{Large algebras of singular functions vanishing on prescribed sets}

\date{}

\author[Bernal]{Luis Bernal-Gonz\'alez}
\address{Departamento de An\'{a}lisis Matem\'{a}tico,\newline\indent Facultad de Matem\'{a}ticas, \newline\indent Universidad de Sevilla,\newline\indent 41080 Avda.~Reina Mercedes s/n, Spain.}
\email{lbernal@us.es}

\author[Calder\'on]{Mar\'{\i}a del Carmen Calder\'on-Moreno}
\address{Departamento de An\'{a}lisis Matem\'{a}tico,\newline\indent Facultad de Matem\'{a}ticas, \newline\indent Universidad de Sevilla,\newline\indent 41080 Avda.~Reina Mercedes s/n, Spain.}
\email{mccm@us.es}

\keywords{Nowhere analytic function, lineability, algebrability, Pringsheim-singular function, zero set}
\subjclass[2010]{15A03, 26B05}

\thanks{The authors have been partially supported by the Plan
Andaluz de Investigaci\'on de la Junta de Andaluc\'{\i}a FQM-127
Grant P08-FQM-03543 and by MEC Grant MTM2012-34847-C02-01.}

\begin{abstract}
\noindent In this paper, the non-vacuousness of the family of all nowhere analytic infinitely differentiable functions on the real line vanishing on a prescribed set \,$Z$ \,is characterized in terms of \,$Z$. In this case, large algebraic structures are found inside such family. The results obtained complete or extend a number of previous ones by several authors.
\end{abstract}

\maketitle

\section{Introduction}

\quad This paper intends to be a contribution to the study of the linear structure of the family of singular functions
on the real line \,$\R$. A singular function is an infinitely differentiable function that is nowhere analytic.
The search for linear --or, in general, algebraic-- structures
within nonlinear sets has become a trend in the last two decades; see, e.g., the survey \cite{bernalpellegrinoseoane2014} or the forthcoming book
\cite{aronbernalpellegrinoseoane2015}.
Here we focus on those singular functions taking the value \,$0$ \,on a prescribed subset of \,$\R$.

\vskip .15cm

Let us now fix some notation, part of it being standard. The symbol \,$\N$ ($\Q$, resp.) will stand for the set of positive integers (the set of rational numbers, resp.), and \,$\N_0 := \N \cup \{0\}$. We also set \,$\aleph_0 = {\rm card} (\N )$ \,and \,$\mathfrak{c} = {\rm card} (\R )$. We will write \,$A^0$ \,and \,$\partial A$ \,to mean, respectively, the interior and the boundary of a given subset \,$A \subset \R$.
For every function \,$f:\R \to \R$, its zero-set is denoted by \,${\mathcal Z}_f$, that is,
${\mathcal Z}_f := \{x \in \R: \, f(x) = 0\} = f^{-1}(\{0\})$. Note that \,${\mathcal Z}_f$ \,is a closed subset of \,$\R$ \,as soon as \,$f$ \,is continuous.

\vskip .15cm

By \,$\mathcal C$, ${\mathcal C}^\infty$ \,and \,$\mathcal S$ \,we denote, respectively, the class of continuous functions \,$\R \to \R$, the class of infinitely differentiable (or ``smooth'') functions \,$\R \to \R$
\,and the subclass formed by those smooth functions \,$f$ \,which are analytic at {\it no} point of \,$\R$. The members of \,$\mathcal S$ \,are called {\it singular} or {\it nowhere analytic} functions. In turn, an important subclass of \,$\mathcal S$ \,is the set \,$\mathcal{PS}$ \,of {\it Pringsheim-singular functions,} that is,  the collection of all ${\mathcal C}^\infty$-functions \,$f:\R \to \R$ \,for which each point \,$x_0 \in \R$ \,is a {\it Pringsheim singularity,} which in turn means that the radius of convergence \,$R(f,x_0)$ \,of the Taylor series associated to \,$f$ \,at the point \,$x_0$ \,equals \,$0$. A classical explicit example of a singular function is the following one due to Lerch \cite{lerch1888}:
$$
f(x) = \sum_{n=1}^\infty {\cos (a^nx) \over n!}
$$
(where \,$a \ge 3$ \,is an odd positive integer),
while an explicit example of a Pringsheim-singular function (see \cite{bernal2008}) is
$$
g(x) = \sum_{n=1}^\infty b_n^{1-n} \sin (b_n x),
$$
where \,$b_n := 2(2 + c_n + (c_{n-1} + \sum_{j=1}^{n-1} b_j^{n+1-j}))$, $c_n := (n+1)!(n+1)^{n+1}$ \,and the term inside the
inner parentheses in the expression of \,$b_n$ \,is defined as \,$0$ \,if \,$n=1$. Of course, $g$ \,is also in \,${\mathcal S}$.

\vskip .15cm

Now, for every set \,$Z \subset \R$, we denote by \,${\mathcal C}_Z$ (${\mathcal C}^\infty_Z$, $\mathcal{S}_Z$, $\mathcal{PS}_Z$, resp.) the subset of all functions \,$f \in {\mathcal C}$ ($f \in {\mathcal C}^\infty$, $f \in \mathcal{S}$, $f \in \mathcal{PS}$, resp.) for which \,${\mathcal Z}_f \subset Z$.

\vskip .15cm

We endow the vector space \,$\mathcal C$ (${\mathcal C}^\infty$, resp.) with the topology of uniform convergence (of uniform convergence of functions and derivatives, resp.) on any compact subset of \,$\R$. This topology makes $\mathcal C$ (${\mathcal C}^\infty$, resp.) an F-space, that is, a completely metrizable topological vector space. The space \,$\mathcal C$ (${\mathcal C}^\infty$, resp.) \,is, in addition, an algebra, from which each subset \,${\mathcal C}_Z$ (${\mathcal C}^\infty_Z$, resp.) \,is a closed subalgebra.

\vskip .15cm

Section 2 is devoted to recall some lineability notions and briefly
review the main known results about the algebraic structure of special subfamilies of \,$C^\infty$,
including the set of singular functions. In Section 3 we provide concrete examples of such functions,
with special emphasis on those ones vanishing on a prescribed subset of \,$\R$. 
Section 4 is the main one and
contains our new statements on the linear structure
of these families.
For the sake of simplicity, we have formulated our findings for functions defined on the whole real line. Nevertheless, with obvious modifications,
they can be stated for functions defined on a fixed interval of \,$\R$. Analogously, we will make use of (or will recall) a number of results that have originally been proved in an interval (mainly the unit interval \,$[0,1]$), but which can be shown to hold in \,$\R$.

\section{Lineability notions and known results}

\quad The emergent theory of lineability has provided a number of concepts in order to quantify the existence of linear or algebraic structures inside a --not necessarily linear-- set. Specifically, in \cite{arongurariyseoane2005,bernal2010,gurariyquarta2004,aronperezseoane2006,bartoszewiczglab2013} the notions given in the next definition were coined; see also \cite{aronbernalpellegrinoseoane2015,bernalpellegrinoseoane2014}.

\begin{definition}
{\rm
Let \,$A$ \,be a subset of a vector space \,$X$, and \,$\alpha$ \,be a cardinal number. Then \,$A$ \,is said to be:
\begin{enumerate}
\item[$\bullet$] {\it lineable} if there is an infinite dimensional vector space \,$M$ \,such that \,$M \setminus \{0\} \subset A$,
\item[$\bullet$] {\it $\alpha$-lineable} if there exists a vector space \,$M$ with \,dim$(M) = \alpha$ \,and \,$M \setminus \{0\} \subset A$, and
\item[$\bullet$] {\it maximal lineable} in \,$X$ \,if \,$A$ \,is ${\rm dim}\,(X)$-lineable.
\end{enumerate}
If, in addition, $X$ \,is a topological vector space, then \,$A$ is \,said to be:
\begin{enumerate}
\item[$\bullet$] {\it dense-lineable} in \,$X$
whenever there is a dense vector subspace \,$M$ \,of \,$X$ \,satisfying \,$M \setminus \{0\} \subset A$, and
\item[$\bullet$] {\it maximal dense-lineable} in \,$X$
whenever there is a dense vector subspace \,$M$ \,of \,$X$ \,satisfying \,$M \setminus \{0\} \subset A$ \,and
dim$\,(M) =$ dim$\,(X)$.
\end{enumerate}
Now, assume that \,$X$ \,is a topological vector space contained in some (linear) algebra. Then \,$A$ \,is called:
\begin{enumerate}
\item[$\bullet$] {\it algebrable} if there is an algebra \,$M$ \,so
    that \,$M \setminus \{0\} \subset A$ \,and \,$M$ \,is infinitely generated, that is, the cardinality of any system of generators of \,$M$ is infinite.
\item[$\bullet$] {\it densely algebrable} in \,$X$ \,if, in addition, $M$ \,can be taken dense in \,$X$.
\item[$\bullet$] {\it $\alpha$-algebrable} if there is an $\alpha$-generated algebra \,$M$ with \,$M \setminus \{0\} \subset A$.
\item[$\bullet$] {\it densely $\alpha$-algebrable} \,if, in addition, $M$ \,can be taken dense in \,$X$.
\item[$\bullet$] {\it strongly $\alpha$-algebrable} if there exists an $\alpha$-generated {\it free} algebra \,$M$ with \,$M \setminus \{0\} \subset A$ (for $\alpha = \aleph_0$, we simply say {\it strongly algebrable}), and
\item[$\bullet$] {\it densely strongly $\alpha$-algebrable} if, in addition, the free algebra \,$M$ can be taken dense in $X$.
\end{enumerate}
}
\end{definition}

Note that if \,$X$ \,is contained in a commutative algebra then a set \,$B \subset X$ \,is a generating set of some free algebra contained in \,$A$
\,if and only if for any \,$N \in \N$, any nonzero polynomial \,$P$ \,in \,$N$ \,variables without constant term and any distinct \,$f_1, \dots ,f_N \in B$, we have \,$P(f_1, \dots ,f_N) \ne 0$ \,and $P(f_1, \dots ,f_N) \in A$. Observe that strong $\alpha$-algebrability \,$\Longrightarrow$ \,$\alpha$-algebrability \,$\Longrightarrow$ \,$\alpha$-lineability. But none of these implications can be reversed; see \cite[p.~74]{bernalpellegrinoseoane2014}.

\vskip .15cm

In 1954, Morgenstern \cite{morgenstern1954} proved that, in a topological sense, ${\mathcal S}$ \,is a huge part of \,${\mathcal C}^\infty$. To be more precise, he showed that \,${\mathcal S}$ \,is a residual subset of the F-space \,${\mathcal C}^\infty$ (see also \cite{bernal1987,darst1973,ramsamujh1991}). One year later, Salzmann and Zeller \cite{salzmannzeller1955} established that, in fact, the smaller subset \,$\mathcal{PS}$ \,is residual in the same F-space. Despite \,$\mathcal S$ \,is clearly not a linear space, Cater \cite{cater1984} showed in 1984 that it is $\mathfrak{c}$-lineable (that is, maximal lineable in \,${\mathcal C}^\infty$). Finally, improvements in the lineability of these sets were established in 2008 by Bernal \cite{bernal2008} (see also \cite{bernalordonez2014}) and in 2014 by Bartoszewicz {\it et al.}~\cite{bartoszewiczbieniasfilipczakglab2014}, as stated respectively in parts (a)--(b) and (c) of the following theorem.

\begin{theorem} \label{Thm-Bernal-Bartoszewicz}
\begin{enumerate}
\item[\rm (a)] The set \,$\mathcal{PS}$ \,is dense-lineable \,in \,${\mathcal C}^\infty$.
\item[\rm (b)] The set \,$\mathcal{S}$ \,is maximal dense-lineable \,in \,${\mathcal C}^\infty$.
\item[\rm (c)] The set \,$\mathcal{S}$ \,is densely strongly $\mathfrak{c}$-algebrable in \,$\mathcal{C}$.
\end{enumerate}
\end{theorem}

Concerning {\it zeros} of smooth functions and lineability, the next theorem gathers two remarkable, recent results that can be found, respectively, in the papers \cite{conejerojimemezmunozseoane2014} and \cite{conejeromunozmurilloseoane2015} by Conejero {\it et al.}~(for large --topological or linear-- size of families formed by just continuous functions, with many zeros, see for instance \cite{dominguez1986} and \cite{enflogurariyseoane2014}).

\begin{theorem}\label{monthly-belgian}
\begin{enumerate}
\item[\rm (a)] The subset of functions in \,$\mathcal{S}$ \,having infinitely many zeros is $\mathfrak{c}$-algebrable.
\item[\rm (b)] The subset of functions in \,$\mathcal{C}^\infty$ having an uncountable set of zeros is strongly $\mathfrak{c}$-algebrable.
\end{enumerate}
\end{theorem}

We will need the algebrability criterion given by the next lemma, whose content was given by Balzerzak {\it et al.}~in \cite[Proposition 7]{balcerzakbartoszewiczfilipczak2013} (see also \cite[Theorem 1.5]{bartoszewiczbieniasfilipczakglab2014})
for a family ${\mathcal F}$ of functions \,$[0,1] \to \R$ \,and \,$a=0$, and then slightly generalized in \cite{bernal2014}
(for \,$a = 0$ \,as well).


\begin{lemma} \label{lemma-exponentials}
Let \,$\Omega$ be a nonempty set, ${\mathcal F}$ be a collection of functions $ \Omega \to \R$ \,and \,$a \in \R$. Assume that there exists a function \,$f:\Omega \to \R$
\,such that \,$f(\Omega )$ \,has some accumulation point in \,$\R$  and \,$\varphi  \circ f \in {\mathcal F}$ \,for every \,$\varphi$ \,belonging to the algebra generated by the functions \,$E_r (x) := e^{rx} + a$ \,$(r \in \R \setminus \{0\})$.
Then \,${\mathcal F}$ \,is strongly $\mathfrak{c}$-algebrable. More precisely, if \,$H \subset (0,+\infty )$ \,is a set with
\,{\rm card}$(H) = \mathfrak{c}$ \,and linearly independent over the field \,$\Q$, then
$$
\{E_r \circ f: \, r \in H\}
$$
is a free system of generators of an algebra contained in \,${\mathcal F} \cup \{0\}$.
\end{lemma}

\begin{proof}
The proof follows the lines of the corresponding proofs in \cite[Proposition 7]{balcerzakbartoszewiczfilipczak2013} and \cite{bernal2014}, but we provide it for the convenience of the reader. Each element of the algebra \,$\mathcal A$ \,generated by the \,$E_r$'s has the shape \,$Q = P(E_{r_1}, \dots , E_{r_p})$ \,for some \,$p \in \N$, some different \,$r_1, \dots ,r_p \in H$ \,and some real polynomial \,$P$ \,in \,$p$ \,variables without constant term and \,$P \not\equiv 0$. Then \,$Q(x)$ \,is a finite sum of terms of the form \,$A (e^{r_1x} + a)^{m_1} \cdots (e^{r_px} + a)^{m_p}$, where $A \ne 0$ \,and \,$m_1, \dots ,m_p \in \N_0$ \,but \,$m_1 + \cdots + m_p > 0$. Each of these terms possesses, in turn, the shape
$$
A \, e^{(m_1r_1 + \cdots + m_pr_p)x} + S(x),
$$
where \,$S$ \,is a finite linear combination of functions of the form \,$e^{bx}$ \,such that \,$b < m_1r_1 + \cdots + m_pr_p$. The $\Q$-independence of the \,$r_i$'s \,implies that all combinations \,$m = m_1r_1 + \cdots + m_pr_p$ \,appearing in the expansion of \,$Q$ \,are mutually different,
so that no cancellation is possible when summing up all terms \,$A e^{mx}$.
By considering the maximum of these \,$m$'s, we obtain \,$|Q(x)| \to +\infty$ \,as \,$x \to +\infty$, hence \,$Q \not\equiv 0$. This shows that \,$\mathcal A$ \,is a free algebra, and \,$\mathcal F$ \,is strongly $\mathfrak{c}$-algebrable.

Now, let \,$\mathcal B$ \,be the algebra generated by the functions
\,$E_r \circ f$ \,$(r \in H )$. Each member of \,$\mathcal B$ \,has the form
$$T = P(E_{r_1} \circ f, \dots , E_{r_p} \circ f),$$
with \,$P$ \,a nonzero polynomial and the \,$E_{r_i}$'s \,as above. Therefore \,$T = Q \circ f$, with \,$Q$ \,as above. If \,$T \equiv 0$, then \,$Q(x) = 0$ \,for all \,$x \in f(\Omega )$. Since this set has some accumulation point in \,$\R$ \,and \,$Q$ \,is analytic in \,$\R$, the Identity Principle tells us that \,$Q \equiv 0$. But this implies \,$P \equiv 0$, so showing the freedom of the system \,$\{E_r \circ f: \, r \in H\}$.
\end{proof}

\section{Set of zeros of special smooth functions}

\quad To start with, in this section we are going to establish a number of propositions telling us under what conditions a smooth function exists having a {\it prescribed} set of zeros. The first of them is well known, but it has been incorporated for the sake of convenience. Recall that, inside \,$\mathcal{C}^\infty$\hskip -2pt , ``singular function'' and ``entire function'' are diametrally opposite concepts: a function \,$f: \R \to \R$ \,is entire if there is \,$x_0 \in \R$ \,such that \,$\sum_{n =0}^\infty {f^{(n)} (x_0) \over n!} \, (x - x_0)^n = f(x)$ \,on all of \,$\R$ \,(equivalently, if for all \,$x_0 \in \R$ \,the same equality holds in \,$\R$).

\begin{proposition}\label{Propexistsentire}
Let \,$Z$ \,be a subset of \,$\R$. The following properties are equivalent:
\begin{enumerate}
\item[\rm (a)] Either \,$Z = \R$ \,or \,$Z$ \,has no accumulation point in \,$\R$.
\item[\rm (b)] There exists an entire function \,$f:\R \to \R$ \,such that \,${\mathcal Z}_f = Z$.
\end{enumerate}
\end{proposition}

\begin{proof}
Since an entire function \,$f$ \,is in particular analytic in \,$\R$, the Identity Principle tells us that either \,$f \equiv 0$ \,or \,${\mathcal Z}_f$ \,has no accumulation point in \,$\R$; this shows that (b) implies (a). As for the reverse implication, assume that \,$Z \subset \R$ \,is as in (a). If \,$Z = \R$ ($Z = \varnothing$, $Z = \{a_1, \dots ,a_N\}$ is finite, resp.) then simply take \,$f \equiv 0$ ($f \equiv 1$, $f(x) = \prod_{j=1}^N (x-a_j)$, resp.). Otherwise, $Z$ \,is a countable infinite set, say \,$\{a_1, \dots , a_n, \dots \}$, with \,$|a_n| \to +\infty$. Passing to the complex plane \,$\C$, the Weierstrass factorization theorem (see for instance \cite{rudin1987}) guarantees the existence of an entire function \,$g : \C \to \C$ \,with \,${\mathcal Z}_g = Z$, namely
$$
f(z) = \prod_{n=1}^\infty \left( 1 - {z \over a_n} \right) \, \exp \left( \sum_{j=1}^n {1 \over j} \, \big( {z  \over a_n} \big)^j \right) ,
$$
with the $nth$-factor replaced by $z$ if $a_n=0$. Finally, choose \,$f := g|_{\R}$.
\end{proof}

The next family of functions will be useful later. For reals \,$a,b$ \,with \,$a < b$, we define \,$\Phi_{a,b} : \R \to \R$, $\Phi_{-\infty ,b}$ \,and \,$\Phi_{a,+\infty}$ \,by
$$
\Phi_{a,b}(x) =  \left\{
\begin{array}{ll}
                 e^{- {1 \over (x-a)^2} - {1 \over (x-b)^2}}  & \mbox{if } x \in (a,b)  \\
                 0 & \mbox{otherwise,}
\end{array} \right.  \eqno (1)
$$
$$
\Phi_{-\infty,b}(x) =  \left\{
\begin{array}{ll}
                 e^{- {1 \over (x-b)^2}}  & \mbox{if } x < b  \\
                 0 & \mbox{otherwise}
\end{array} \right. \hbox{\,and\,\,\,\,}
\Phi_{a,+\infty}(x) =  \left\{
\begin{array}{ll}
                 e^{- {1 \over (x-a)^2}}  & \mbox{if } x > a  \\
                 0 & \mbox{otherwise.}
\end{array} \right.  \eqno (2)
$$

\begin{remark}\label{remark1}
{\rm The following facts are well known (or easy to check):
\begin{itemize}
\item Each $\Phi_{a,b}$ ($\Phi_{-\infty ,b}$, $\Phi_{a,+\infty}$) \,is in \,$\mathcal{C}^\infty$\,and is analytic in \,$\R$ \,except at \,$a$ \,and \,$b$ (except at \,$b$, except at \,$a$, resp.).
\item $\Phi_{a,b}^{(k)} (a) = 0 = \Phi_{a,b}^{(k)} (b) = \Phi_{-\infty ,b}^{(k)} (b) = \Phi_{a,+\infty}^{(k)} (a)$ \,for every \,$k \in \N_0$.
\item For all \,$x \in (a,b)$ (all \,$x \in (-\infty ,b)$, all \,$x \in (a,\infty )$, resp.), we have, respectively
$$
\Phi_{a,b}^{(k)} (x) = P_k(x-a,x-b) (x-a)^{-3k} (x-b)^{-3k} e^{- (x-a)^{-2} - (x-b)^{-2}},
$$
$$
\Phi_{-\infty,b}^{(k)} (x) = P_k(x-b) (x-b)^{-3k} e^{- (x-b)^{-2}},
$$
$$
\Phi_{a,+\infty}^{(k)} (x) = P_k(x-a) (x-a)^{-3k} e^{- (x-a)^{-2}},
$$
where \,$P_k(\cdot , \cdot )$ ($P_k(\cdot )$, resp.) is a symmetric polynomial in two variables (is a polynomial in one variable, resp.) depending only on \,$k$.
\end{itemize}}
\end{remark}

\begin{proposition}\label{Propexistssmooth}
Let \,$Z$ \,be a subset of \,$\R$. The following properties are equivalent:
\begin{enumerate}
\item[\rm (a)] $Z$ \,is closed.
\item[\rm (b)] There exists a function \,$f \in {\mathcal C}^\infty$ \,such that \,${\mathcal Z}_f = Z$.
\end{enumerate}
\end{proposition}

\begin{proof}
Trivially, (b) implies (a), because \,${\mathcal Z}_f = f^{-1}(\{0\})$ \,and every smooth function is continuous. In order to prove that (a) implies (b), fix a nonempty proper closed subset \,$Z \subset \R$ (the case \,$Z \in \{\varnothing , \R \}$ \,is trivial). Then \,$\R \setminus Z$ \,is a nonempty proper open subset of \,$\R$, so that there is \,$M \subset \N$ (with \,$M = \{1, \dots ,N\}$ \,or \,$M = \N$) \,such that \,$\R \setminus Z$ \,can be written as a countable disjoint union
$$
\R \setminus Z = \bigcup_{n \in M} (a_n,b_n),
$$
where \,$-\infty \le a_n < b_n \le +\infty$ (of course, at most one \,$a_n$ \,is \,$-\infty$ \,and at most one \,$b_n$ \,is \,$+\infty$). Define the function
\,$\R \to \R$ \,by
$$
f = \sum_{n \in M} \Phi_{a_n,b_n},
$$
with the \,$\Phi_{a_n,b_n}$'s \,given by (1) and (2).
The following is evident: $f$ \,is well defined, ${\mathcal Z}_f = Z$, $f$ \,has finite derivatives of all orders at each \,$x \in (\R \setminus Z) \cup Z^0$
\,and \,$f^{(k)}(x) = 0$ \,for all \,$x \in Z^0$ \,and all \,$k \in \N$.

\vskip .15cm

It remains only to prove that \,$f$ \,is infinitely derivable at each \,$c \in \partial Z$. In turn, it is enough to show that \,$f_+^{(k)}(c) = 0 = f_-^{(k)} (c)$ $(k \in \N)$ \,for such a point \,$c$. Let us show that \,$f_+^{(k)}(c) = 0$ \,for all \,$k$, the proof of \,$f_-^{(k)}(c) = 0$ \,being analogous. Consider first \,$k=1$. Three cases are possible: either \,$c = a_m$ \,for some \,$m$, or \,$[c,c+\delta ) \subset Z$ \,for some \,$\delta > 0$, or for any \,$\delta > 0$ \,the intersection \,$Z \cap (c,c + \delta )$ \,is nonempty and infinitely many \,$(a_n,b_n)$ \,are contained in \,$(c,c + \delta )$. In the first case, $f_+'(c) = 0$ \,by Remark \ref{remark1}, while the equality is plain in the second case. In the third case, fix \,$\ve > 0$. Since $\lim_{t \to 0} t^{-1} e^{-1/t^2} = 0$, we can choose \,$\delta > 0$ \,so that \,$t^{-1} e^{-1/t^2} < \ve$ \,if \,$0 < t < \delta$. Given \,$x \in (c,c + \delta )$,
one has that either \,$x \in Z$ \,or there is a unique \,$m = m(x) \in \N$ \,such that \,$c < a_m < x < b_m$. Then the quotient \,${f(x) - f(c) \over x-c}$ \,is either \,$0$ \,or
$$(x-c)^{-1} \Phi_{a_m,b_m} (x) = {x-a_m \over x-c}(x-a_m)^{-1} e^{- {1 \over (x-a_m)^2}} e^{- {1 \over (x-b_m)^2}},$$
and this expression is (positive and) less than
\,$(x-a_m)^{-1} e^{- {1 \over (x-a_m)^2}}$. Since \,$0 < x - a_m < \delta$, the mentioned expression is \,$< \ve$.
We conclude that \,$f_+'(c) = \lim_{x \to c^+}{f(x) - f(c) \over x-c} = 0$. Finally, an induction procedure (on the order \,$k$) using the formulas for derivatives given in Remark \ref{remark1} together with the property \,$\lim_{t \to 0} t^{-N} e^{-1/t^2} = 0$ $(N \in \N)$ \,completes the proof.
\end{proof}

\begin{remark} \label{remark2}
{\rm Note that the function constructed in the last proof satisfies, in addition, that \,$f$ \,is analytic on \,$(\R \setminus Z) \cup Z^0$ \,and \,$f^{(n)}|_Z = 0$ \,for all \,$k \in \N_0$.}
\end{remark}

\begin{proposition}\label{Propexistssingular}
Let \,$Z$ \,be a subset of \,$\R$. The following properties are equivalent:
\begin{enumerate}
\item[\rm (a)] $Z$ \,is closed and \,$Z^0 = \varnothing$.
\item[\rm (b)] There exists a function \,$f \in {\mathcal S}$ \,such that \,${\mathcal Z}_f = Z$.
\end{enumerate}
\end{proposition}

\begin{proof}
Again, if (b) holds then \,$Z$ \,is the pre-image of \,$\{0\}$ \,under a continuous function, and hence it must be closed.
If \,$Z^0 \ne \varnothing$ \,then there would be an interval \,$(x_0 - \delta , x_0 + \delta )$ \,on which \,$f = 0$, so yielding the analyticity of \,$f$ \,at \,$x_0$, which contradicts the hypothesis. Then (a) is derived from (b).

\vskip .15cm

As for the reverse implication, assume that \,$Z$ \,is closed and \,$Z^0 = \varnothing$. If $Z= \varnothing$, we fix any bounded \,$\psi \in \mathcal{S}$, for instance the Lerch function \,$\sum_{n=1}^\infty {\cos (3^nx) \over n!}$ (see Section 1; another example can be found in \cite{kimkwon2000}). Then the function \,$\varphi (x) := \psi (x) + \alpha$, where \,$\alpha := 1 + \sup_{x \in \R} |\psi (x)|$, belongs to \,$\mathcal{S}$ \,and \,$\mathcal{Z}_{\varphi} = \varnothing$.

\vskip .15cm

Assume now that \,$Z \ne \varnothing$. According to the last proposition and Remark \ref{remark2}, there is \,$g \in \mathcal{C}^\infty$ \,such that \,$\mathcal{Z}_g = Z$ \,and \,$g$ \,is analytic at each point of \,$\R \setminus Z$. Therefore the function \,$f := \varphi \cdot g$ \,is also in \,$\mathcal{C}^\infty$ \,and satisfies \,$\mathcal{Z}_f = Z$. Our only task is to show that \,$f$ \,is nowhere analytic. Suppose, by way of contradiction, that \,$f$ \,is analytic at some point. Since the set of points of analyticity is always open, and because of \,$Z^0 = \varnothing$, we get that there is an interval \,$J \subset \R \setminus Z$ \,where \,$f$ \,is analytic. Now, as \,$g (x) \ne 0$ \,for all \,$x \in J$, the reciprocal \,$1/g$ \,is also analytic on \,$J$, hence the product \,$f \cdot {1 \over g}$ \,is too. But \,$f \cdot {1 \over g}= \varphi$, which is absurd. The proof is complete.
\end{proof}

\begin{remark} \label{remark3}
Observe that, according to the last proposition, one can get smooth nowhere analytic functions having zeros in sets that are as large in (Lebesgue) measure as one desires. Indeed, given \,$\ve > 0$, choose \,$Z := \R \setminus \bigcup_{n=1}^\infty (q_n - {\ve \over 2^{n+1}},q_n + {\ve \over 2^{n+1}})$, where \,$(q_n)$ \,is an enumeration of \,$\Q$.
\end{remark}

In the case of Pringsheim-singular functions, we have not been able to obtain a cha\-rac\-te\-ri\-za\-tion of zero-sets. Nevertheless, at least discrete sets are among such zero-sets.

\begin{proposition}\label{PropexistsPringsheim}
Let \,$Z$ \,be a subset of \,$\R$ \,without accumulation points. Then there exists a function \,$f \in \mathcal{PS}$ \,such that \,${\mathcal Z}_f = Z$.
\end{proposition}

\begin{proof}
Similarly to the proof of Proposition \ref{Propexistssingular}, take a bounded member of \,$\mathcal{PS}$, as for instance the function
\,$g(x) = \sum_{n=1}^\infty b_n^{1-n} \sin (b_n x)$ \,given in Section 1 (with the \,$b_n$'s \,large, as described there). Then the translated function
\,$h(x) := g(x) + 1 + \sup_{\R} |g|$ \,is also Pringsheim-singular and \,$\mathcal{Z}_h = \varnothing$. According to Proposition \ref{Propexistsentire}, there exists an entire function \,$\varphi : \R \to \R$ \,having zeros exactly at the points of \,$Z$. Define the function \,$f:\R \to \R$ \,by
$$
f := h \cdot \varphi .
$$
Plainly, $f \in \mathcal{C}^\infty$ \,and \,$\mathcal{Z}_f = Z$. All that should be proved if that \,$R(f,x_0) = 0$ \,for any given \,$x_0 \in \R$; recall that
\,$R(f,x_0) = (\limsup_{n \to \infty} |f^{(n)}(x_0)/n!|^{1/n})^{-1}$. In order to achieve this, we invoke the theory of formal power series, which can be found, for instance in \cite[Chap.~1]{cartan 1997}. Since the zeros of a nonzero analytic function are isolated, we have that either \,$\varphi (x_0) \ne 0$ \,or \,$x_0$ \,is a zero of \,$\varphi$ \,of certain (finite) order. Altogether, there is \,$N \in \N_0$ \,and \,$\psi$ \,analytic in \,$\R$ \,(in fact, entire) such that \,$\psi(x_0) \ne 0$ (hence \,$\psi(x) \ne 0$ \,for all \,$x$ \,in some neighborhood \,$J$ \,of \,$x_0$) \,and \,$f(x) = (x-x_0)^N \psi (x) h(x)$ \,for all \,$x \in \R$. Then
$$h_0(x) = f(x) \cdot {1 \over \psi (x)} \quad \hbox{for all } x \in J,$$
where we have set \,$h_0 (x) := (x-x_0)^N h(x)$. Assume, by way of contradiction, that \,$R(f,x_0) > 0$. Since \,$1/\psi$ \,is analytic at \,$x_0$, we get \,$R(1/\psi ,x_0) > 0$. According to the theory of formal power series, we obtain \,$R(h_0,x_0) = R(f \cdot (1/\psi ),x_0) > 0$. But \,$R(h,x_0) = R(h_0,x_0)$, because \,${h_0^{(n)}(x_0) \over n!} = {h^{(n-N)}(x_0) \over (n-N)!}$ \,for all \,$n \ge N$. Therefore \,$R(h,x_0) > 0$, which is absurd.
\end{proof}

\begin{remark} \label{preservesingular}Observe that the following is obtained from the last proof: If $f \in \mathcal{PS}_Z$, then $\varphi \cdot f \in \mathcal{PS}_Z$ for any nonzero entire function $\varphi : \R \to \R$.
\end{remark}

\section{Large algebras of singular functions}
In the above section, we were able to prescribe the set $Z_f$ of zeros of a function $f$ satisfying special properties. However, there is {\it no} hope to find a large linear space \,$M$ \,of functions \,$\R \to \R$ \,with {\it exactly} the same prescribed subset of zeros for all \,$f \in M \setminus \{0\}$.
\begin{proposition} Let $Z\subset \R $ and $M$ be a linear space of functions $f: \R \to \R$ such that $Z_f=Z$ for all $f \in M\setminus \{ 0\}$. Then dim\,$(M)$\,$\leq 1$.
\end{proposition}
\begin{proof}
The case \,$Z = \R$ is trivial. Assume that $Z \ne \R$ and $f,g \in M$ \,are linearly independent and choose any \,$x_0 \not\in Z$. Then \,$f(x_0) \ne 0 \ne g(x_0)$. It follows that \,$h := g(x_0)f - f(x_0)g \in M \setminus \{0\}$ \,but \,$h(x_0) = 0$, so \,$\mathcal{Z}_h \ne Z$.
\end{proof}
\vskip .15cm

Although it is not possible to fix the exact set of zeros, we can expect that $Z \subset Z_f$ for every $f \in M \setminus \{0\}$. The following collection of theorems characterize the algebrability of ${\mathcal C}^\infty_Z$ and ${\mathcal S}_Z$. This improves and completes the results contained in Theorems \ref{Thm-Bernal-Bartoszewicz} and \ref{monthly-belgian}.

\begin{theorem}\label{Theoremsmooth}
Let \,$Z$ \,be a subset of \,$\R$. The following properties are equivalent:
\begin{enumerate}
\item[\rm (a)] $Z$ \,is not dense in \,$\R$.
\item[\rm (b)] ${\mathcal C}^\infty_Z \ne \{0\}$.
\end{enumerate}
If this is the case, then \,${\mathcal C}^\infty_Z$ \,is strongly $\mathfrak{c}$-algebrable.
\end{theorem}

\begin{proof}
The equivalence of (a) and (b) follows directly from Proposition \ref{Propexistssmooth}. Now, take any function \,$f \in {\mathcal C}^\infty_Z \setminus \{0\}$ \,and apply Lemma \ref{lemma-exponentials} \,with \,$\Omega =\R$, $a = -1$ \,and \,$\mathcal{F} = {\mathcal C}^\infty_Z$ (note that \,$f(\Omega )$ \,is a non-degenerate interval, so it has some accumulation point). This furnishes a free $\mathfrak{c}$-generated algebra \,$\mathcal{B} \subset \mathcal{C}^\infty$. Moreover, each generating function \,$e^{r f(x)} -1$ \,vanishes on \,$Z$, so each member of \,$\mathcal B$ \,also does. In other words, $\mathcal{B} \subset {\mathcal C}^\infty_Z$, which proves the theorem.
\end{proof}

\begin{theorem}\label{Theoremsingular}
Let \,$Z$ \,be a subset of \,$\R$. The following properties are equivalent:
\begin{enumerate}
\item[\rm (a)] $Z$ \,is nowhere dense in \,$\R$.
\item[\rm (b)] ${\mathcal S}_Z \ne \varnothing$.
\end{enumerate}
If this is the case, then \,${\mathcal S}_Z$ \,is strongly $\mathfrak{c}$-algebrable.
\end{theorem}

\begin{proof}
The equivalence of (a) and (b) follows directly from Proposition \ref{Propexistssingular}.
Let us prove the strong $\mathfrak{c}$-algebrability of  \,${\mathcal S}_Z$.
Similarly to the proof of Theorem \ref{Theoremsmooth},
take any function \,$f \in {\mathcal S}_Z$ \,and apply Lemma \ref{lemma-exponentials} \,with \,$\Omega =\R$, $a = -1$ \,and
\,$\mathcal{F} = {\mathcal S}_Z$.  This provides a free $\mathfrak{c}$-generated algebra \,$\mathcal{B} \subset \mathcal{C}^\infty$ \,all of whose members vanish on \,$Z$.

\vskip .15cm

It remains to prove that every \,$F \in \mathcal{B} \setminus \{0\}$ \,belongs to \,$\mathcal S$. Observe that such an \,$F$ \,has the form
$$
F = Q \circ f,
$$
where \,$Q(x) = P(e^{r_1 x} -1, \dots ,e^{r_p x} - 1)$, $P$ \,is a nonzero polynomial without constant term and \,$r_1, \dots ,r_p \in H$, so that \,$r_1, \dots ,r_p$ \,are positive and linearly $\Q$-independent. Then \,$Q$ \,is not identically zero (see the proof of Lemma \ref{lemma-exponentials}). Therefore \,$Q$ \,is analytic in \,$\R$ \,and nonconstant. Hence the set \,$\mathcal{Z}_{Q'} = \{x \in \R : \, Q'(x) = 0\}$ \,has no accumulation points, so it is countable.
Suppose, by way of contradiction, that \,$F \not\in \mathcal{S}$. Since the set of analyticity points is open, there is an open interval \,$J$ \,where \,$F$ \,is analytic. Since \,$f$ \,cannot be constant on any non-degenerate interval, the image \,$f(J)$ \,is again a non-degenerate interval. Therefore \,$f(J) \setminus
\mathcal{Z}_{Q'} \ne \varnothing$. Consequently, there is \,$x_0 \in J$ \,such that \,$Q'(f(x_0)) \ne 0$. Hence (see \cite[Chap.~1]{cartan1997}) there are open intervals \,$C$ \,and \,$D$ \,such that \,$f(x_0) \in C \subset f(J)$, $F(x_0) \in D \subset F(J)$, $Q : C \to D$ \,is bijective and \,$Q^{-1} : D \to C$ \,is analytic. But \,$f = Q^{-1} \circ F$ \,on the open set \,$f^{-1}(C)$ (which is nonempty because it contains \,$x_0$). Finally, the composition of analytic functions is also analytic, so \,$f$ \,should be analytic at \,$x_0$. This is the desired contradiction.
\end{proof}

Concerning the family \,$\mathcal{PS}$, we have done several findings, but they are not as far-reaching as those obtained for \,$\mathcal{S}$.

\begin{theorem}\label{TheoremPringsheimsingular}
\begin{enumerate}
\item[\rm (a)] If \,$Z$ \,is a subset of \,$\R$ \,without accumulation points then \,$\mathcal{PS}_Z$ \,is maximal lineable in \,$\mathcal{C}^\infty$.
\item[\rm (b)] The set \,$\mathcal{PS}$ \,is maximal dense-lineable in \,${\mathcal C}^\infty$.
\end{enumerate}
\end{theorem}

\begin{proof} (a) According to Proposition \ref{PropexistsPringsheim}, there exists \,$f \in \mathcal{PS}$ \,such that \,$\mathcal{Z}_f = Z$. In particular, $f \in \mathcal{PS}_Z$ and, by Remark \ref{preservesingular}, \,$\varphi f \in \mathcal{PS}_Z$ \,for each nonzero \,$\varphi \in \mathcal{E} := \{$entire functions$\}$. Then the vector space \,$M := \{\varphi f : \, \varphi \in \mathcal{E}\}$ \,is contained in \,$\mathcal{PS}_Z \cup \{0\}$. Now, it is well known  that \,dim$(\mathcal{E}) = \mathfrak{c}$ (for instance, the exponentials \,$x \mapsto e^{cx}$, $c > 0$, are linearly independent). From this and the fact that neither \,a nonzero entire function \,$\varphi$ \,nor \,$f$ \,can vanish on a somewhere dense set it follows that \,dim$(M) = \mathfrak{c} = {\rm dim} (\mathcal{C}^\infty)$, as required.

\vskip .15cm

\noindent (b) Taking \,$Z = \varnothing$ in part (a) yields the existence of a function \,$f \in \mathcal{PS}$ \,having no zeros such that the vector space \,$M = \{\varphi f : \, \varphi \in \mathcal{E}\}$ \,is contained in \,$\mathcal{PS} \cup \{0\}$. Since the polynomials form a dense subset of \,$\mathcal{C}^\infty$ \,and are contained in \,$\mathcal E$, we obtain that \,$\mathcal E$ \,is dense in \,$\mathcal{C}^\infty$. Therefore \,$M = f \cdot \mathcal{E}$ \,is dense in \,$f \cdot \mathcal{C}^\infty$. But \,$f \cdot \mathcal{C}^\infty = \mathcal{C}^\infty$ \,because \,$f$ \,lacks zeros. Altogether, we get that \,$M$ \,is a dense
$\mathfrak{c}$-dimensional linear subspace of \,$\mathcal{PS}$, which proves (b).
\end{proof}

\begin{remark}
{\rm Part (b) of the last theorem is known for {\it complex-valued} smooth functions on \,$\R$ \cite[Theorem 4.7(c)]{bernalordonez2014}, but not (as far as we know) for real functions.}
\end{remark}

\vskip .15cm

In the next and final theorem, algebrability (in a rather high degree) is obtained for the family of smooth functions having Pringsheim singularities at almost all (in a topological sense) points of \,$\R$. To be more precise, we consider the family
\begin{equation*}
\begin{split}
\phantom{aaaaaaaaa} \widetilde{\mathcal{PS}} := \{f \in \mathcal{C}^\infty : \, &\hbox{exists an open dense set } \,G = G(f) \hbox{ \,such that } \,f \\
                                                        &\hbox{has a Pringsheim singularity at each point of } \,G\}. \qquad \quad \,(3)
\end{split}
\end{equation*}
Notice that each set \,$G(f)$ \,in (3), being open and dense, is also residual in \,$\R$.
Notice also that \,$\mathcal{PS} \subset \widetilde{\mathcal{PS}} \subset \mathcal{S}$, where the first inclusion is obtained by taking \,$G(f) = \R$ \,and the second one follows from the fact that the points where a function is analytic is open. Then the following result improves part (c) of Theorem \ref{Thm-Bernal-Bartoszewicz}.

\begin{theorem} \label{Thm-CasiPringsheimsingular}
The set \,$\widetilde{\mathcal{PS}}$ \,is densely strongly $\mathfrak{c}$-algebrable in \,$\mathcal{C}$.
\end{theorem}

\begin{proof}
Similarly to the proof of Theorem \ref{Theoremsingular},
take any function \,$f \in \mathcal{PS}$ \,and apply Lemma \ref{lemma-exponentials} \,with \,$\Omega =\R$, $a = 0$ \,and
\,$\mathcal{F} = \widetilde{\mathcal{PS}}$.  This provides a free $\mathfrak{c}$-generated algebra \,$\mathcal{B} \subset \mathcal{C}^\infty$.

\vskip .15cm

It must be shown that every \,$F \in \mathcal{B} \setminus \{0\}$ \,belongs to \,$ \widetilde{\mathcal{PS}}$. To this end, we follow closely the proof of Theorem \ref{Theoremsingular}. Such a function \,$F$ \,has the form
\,$F = Q \circ f$,
where \,$Q(x) = P(e^{r_1 x}, \dots ,e^{r_p x})$, $P$ \,is a nonzero polynomial without constant term and \,$r_1, \dots ,r_p \in H$, so that \,$r_1, \dots ,r_p$ \,are positive and linearly $\Q$-independent. This forces \,$Q$ \,to be not identically zero. Therefore \,$Q$ \,is analytic in \,$\R$ \,and nonconstant. Hence the set \,$\mathcal{Z}_{Q'} = \{x \in \R : \, Q'(x) = 0\}$ \,has no accumulation points.
Suppose, by way of contradiction, that \,$F \not\in \widetilde{\mathcal{PS}}$. Then there exists an open interval \,$J$ \,such that \,$R(F,x) > 0$ \,for all \,$x \in J$. Since \,$f$ \,is continuous but it cannot be constant on any non-degenerate interval, the image \,$f(J)$ \,is again a non-degenerate interval. Therefore \,$f(J) \setminus
\mathcal{Z}_{Q'} \ne \varnothing$. Consequently, there is \,$x_0 \in J$ \,such that \,$Q'(f(x_0)) \ne 0$.
Now, we invoke again the theory of formal power series.
There are open intervals \,$C$ \,and \,$D$ \,such that \,$f(x_0) \in C \subset f(J)$, $F(x_0) \in D \subset F(J)$, $Q : C \to D$ \,is bijective and \,$Q^{-1} : D \to C$ \,is analytic. In particular, $R(Q^{-1}, F(x_0)) > 0$. But \,$R(F,x_0) > 0$ (because \,$x_0 \in J$) \,and \,$f = Q^{-1} \circ F$ \,in a neighborhood of \,$x_0$. Therefore the formal Taylor series of \,$f$ \,at \,$x_0$ \,is the composition of the formal Taylor series of \,$Q^{-1}$ \,and \,$f$ \,at \,$F(x_0)$ \,and \,$x_0$, respectively. It follows (see \cite[Chap.~1]{cartan1997}) that \,$R(f,x_0) > 0$, which contradicts the fact \,$f \in\mathcal{PS}$.

\vskip .15cm

Finally, as in \cite[Section 6]{bartoszewiczbieniasfilipczakglab2014}, we use a version of the Stone--Weierstrass theorem (see, e.g., \cite{rudin1991}) to prove the density of the algebra \,$\mathcal{B}$ \,in \,$\mathcal{C}$.  According to this theorem, if for given distinct points \,$x_1,x_2 \in \R$ \,one could find \,$\varphi, \, \psi \in \mathcal{B}$ \,such that \,$\varphi (x_1) \ne 0$ \,and \,$\psi (x_1) \ne \psi (x_2)$, then one would have the desired density. But this is easily achieved just by choosing \,$\varphi (x) = \psi (x) = e^{f(x)}$. Indeed, the property \,$\varphi (x_1) \ne 0$ \,is evident, and the required second inequality is feasible because \,$x \mapsto e^x$ \,is injective and the function \,$f$ \,can be chosen to be injective as well. For this, suffice it to choose a strictly positive \,$h \in \mathcal{PS}$ \,(like the one in the beginning of the proof of Proposition \ref{PropexistsPringsheim}) \,and then to select \,$f(x) := \int_0^x h(t) \,dt$. The proof is concluded.
\end{proof}

\begin{remark}
{\rm One may wonder whether the inclusions \,$\mathcal{PS} \subset \widetilde{\mathcal{PS}} \subset \mathcal{S}$ \,are strict. The answer is positive, but in fact most singularities of a nowhere analytic function are Pringsheim. Indeed, a result published by Zahorski \cite{zahorski1947} in 1947 establishes the exact structure of the set of these singularities. Namely, if for each \,$f \in \mathcal{C}^\infty$ \,we call \,$PS(f) = \{x \in \R : \,x$ is a Pringsheim singularity for $f\}$ \,and \,$CS(f) = \{x \in \R \setminus PS(f): \,f$ is not analytic at $x\}$ (the set of ``Cauchy singularities'' of \,$f$) \,then,
given two subsets \,$A,\,B \subset \R$, we have that there exists \,$f \in \mathcal{C}^\infty$ \,such that \,$PS(f) = A$ \,and \,$CS(f) = B$ \,if and only if
\,$A$ \,is a $G_\delta$ subset, $B$ \,is an $F_\sigma$ subset of the first category, $A \cup B$ \,is closed and \,$A \cap B = \varnothing$.
Consequently, if \,$f \in \mathcal{S}$ \,then \,$PS(f) \cup CS(f) = \R$ and, therefore, $PS(f)$ \,is a dense $G_\delta$ subset, hence residual. Despite this, choosing \,$(A = \R \setminus \{0\},B = \{0\})$ \,and \,$(A = \R \setminus \Q ,B = \Q )$ \,provides, respectively, functions
\,$f \in \widetilde{\mathcal{PS}} \setminus \mathcal{PS}$ \,and \,$f \in \mathcal{S} \setminus \widetilde{\mathcal{PS}}$.}
\end{remark}

\begin{remark}
{\rm In \cite{bastinessernikolay2012} the authors study the genericity of \,$\mathcal{S}$ \,in a measure-theoretic sense. In fact, this kind of genericity as well as the residuality in \,$\mathcal{C}^\infty$ \,are shown to be true for the smaller family \,$\mathcal{NG}$ \,of {\it nowhere Gevrey differentiable functions.} By definition, a function \,$f \in \mathcal{C}^\infty$ \,belongs to \,$\mathcal{NG}$ \,provided that,
for every \,$(x_0,s,\delta ,C,h) \in \R \times (1,+\infty ) \times (0,+\infty )^3$, there exists
\,$(x,n) \in (x_0 - \delta , x_0 + \delta ) \times \N_0$ \,such that
$$|f^{(n)}(x)| > C \,h^n (n!)^s.$$
In \cite{bastinconejeroesserseoane2014} Bastin et {\it al.}~established that \,$\mathcal{NG}$ \,is densely $\mathfrak{c}$-algebrable in \,$\mathcal{C}^\infty$.
As a matter of fact, $\mathcal{NG}$ \,is strongly $\mathfrak{c}$-algebrable (see \cite[Chap.~2]{aronbernalpellegrinoseoane2015}).
In particular, this result covers part (b) and complements part (c) of Theorem \ref{Thm-Bernal-Bartoszewicz}.}
\end{remark}

\noindent We want to finish this paper by listing a collection of {\bf open questions:}
\begin{enumerate}
\item[\bf A.] Is \,$\mathcal{PS}$ \,(strongly $\mathfrak{c}$-) algebrable?
\item[\bf B.] In view of Theorems \ref{Thm-Bernal-Bartoszewicz}(c) and \ref{Thm-CasiPringsheimsingular}, is \,$\mathcal{PS}$ \,densely strongly $\mathfrak{c}$-algebrable in \,$\mathcal C$? Do these density properties hold with respect to the space \,$\mathcal{C}^\infty$?
\item[\bf C.] For appropriate sets \,$Z$, are \,$\mathcal{S}_Z$ \,and \,$\mathcal{PS}_Z$ \,{\it densely} (strongly) algebrable in \,$\mathcal{C}_Z$ (or better, in \,$\mathcal{C}^\infty_Z$)?
\item[\bf D.] In the same vein, what can be said about lineability properties of \,$\mathcal{NG}_Z := \{f \in \mathcal{NG}: \,Z \subset \mathcal{Z}_f\}$?
\end{enumerate}

%

\vskip .15cm


\end{document}